\documentclass[12pt,a4paper,reqno,oneside]{amsart}
\usepackage{amsmath,amssymb,amsthm,mathrsfs,calligra,amsfonts}

\usepackage{csquotes}
\usepackage{vmargin}
\setpapersize{A4}
\setmargins
{2.5cm}		
{1.5cm}		
{16cm}		
{23.42cm}	
{10pt}		
{1cm}		
{0pt}		
{2cm}		

\usepackage[dvipsnames]{xcolor}
\usepackage[%
bookmarks=true,			
unicode=true,			
colorlinks=true,		
linkcolor=Blue,			
citecolor=Blue,			
filecolor=magenta,		
urlcolor=RoyalBlue		
]{hyperref}				
\usepackage[all]{xy}	%
  
\CompileMatrices                            

\UseTips                                    

\input xypic
\usepackage[bookmarks=true]{hyperref}       

\usepackage{amssymb,latexsym,amsmath,amscd}
\usepackage{xspace}
\usepackage{color}
\usepackage{graphicx}
\usepackage{dsfont}

\reversemarginpar

\DeclareMathOperator{\rank}{rank}

\theoremstyle{plain}
\newtheorem{theorem}{Theorem}[section]
\newtheorem*{theorem*}{Theorem}
\newtheorem{proposition}[theorem]{Proposition}

\newtheorem{lemma}[theorem]{Lemma}

\theoremstyle{definition}
\newtheorem{definition}[theorem]{Definition}
\newtheorem*{definition*}{Definition}

\newtheorem{remark}[theorem]{Remark}

\newcommand{\enm}[1]{\ensuremath{#1}}          %
\newcommand{\op}[1]{\operatorname{#1}}
\newcommand{\cal}[1]{\mathcal{#1}}

\renewcommand{\bar}[1]{\overline{#1}}

\newcommand{\CC}{\enm{\mathbb{C}}}

\newcommand{\RR}{\enm{\mathbb{R}}}

\newcommand{\ZZ}{\enm{\mathbb{Z}}}

\newcommand{\PP}{\enm{\mathbb{P}}}

\newcommand{\Ff}{\enm{\cal{F}}}

\newcommand{\Ii}{\enm{\cal{I}}}

\newcommand{\Mm}{\enm{\cal{M}}}
\newcommand{\Nn}{\enm{\cal{N}}}
\newcommand{\Oo}{\enm{\cal{O}}}

\renewcommand{\phi}{\varphi}
\renewcommand{\theta}{\vartheta}
\renewcommand{\epsilon}{\varepsilon}

\newcommand{\Pic}{\op{Pic}}

\newcommand{\Hom}{\op{Hom}}
\newcommand{\Ext}{\op{Ext}}

\newcommand{\End}{\op{End}}

\DeclareMathOperator{\red}{red}

\newcommand{\lext}{\mathscr{E}\text{\kern -2pt {\it{xt}}}\,}
\newcommand{\lend}{\mathscr{E}\text{\kern -2pt {\it{nd}}}}
\newcommand{\lhom}{\mathcal{H}om}
\newcommand{\ce}{\mathrel{\mathop:}=}

\renewcommand{\to}[1][]{\xrightarrow{\ #1\ }}

\newcommand{\old}[1]{}

\begin{document}

\date{}
\author{Edoardo Ballico and Elizabeth Gasparim}  
\address{Univ. of Trento, Italy; edoardo.ballico@unitn.it;
ICTP, Italy; etgasparim@gmail.com}
\title{Higher rank bundles on Hopf surfaces}

\maketitle

\begin{abstract} 
We  prove  the existence of filtrable  stable bundles on the Hopf surface with any rank and 
 $c_2>0$, 
and show they must have jumps.
In the opposite direction,  for each integer $r\ge 2$ we prove the existence 
of  irreducible rank $r$ vector bundles on $X$ with trivial determinant, $c_2=1$, and  no jumps.
  
 We  apply  codimension $2$ operations to elements of the moduli space 
$\mathcal M_{r,n}$ of rank $r$ stable vector bundles  with $c_2=n$ on the Hopf surface, obtaining
families of torsion free sheaves in the closure of  $\mathcal M_{r,n+1}$ when $c_2>1$.

\end{abstract}

\tableofcontents

\section{Introduction}

We study higher rank vector bundles  on a classical Hopf surface
in terms of  their stability, filtrability, and jumping behaviour. We discuss some features 
of the moduli spaces $\mathcal M_{r,n}$ of  rank $r\geq 2$ bundles with 
$c_2=n$ and  construct
 torsion free sheaves which are in their closure. 
Such torsion free sheaves may be interpreted in terms of 
{ instanton bubbling}, because  they correspond to the algebro-geometric version
of the phenomenon in differential geometry where a sequence of  instanton connections degenerates 
  forming point instantons, that is, instantons  concentrated  at points.
  In other words, the moduli of vector bundles $\mathcal M_{r,n}$ may be regarded as  instanton moduli 
  spaces having point instantons in their closure, described by Theorem\thinspace\ref{boundary} 
  in terms of
  torsion free  sheaves which
 are the limits of families  of generic elements (stable and locally free) 
 of the moduli space $\mathcal M_{r,n}$.

 In section \ref{hf} we give basic definitions of stability and jumping behaviour of bundles 
 on Hopf surfaces.
 
  In section \ref{Sf} we prove 
that filtrable  bundles on a classical Hopf surface having $c_2(E)>0$ must have jumps (Theorem\thinspace\ref{f00}).
We then prove the  existence of stable and filtrable rank $r$ vector bundles on $X$ 
such that $\det(E)\cong \Oo_X$ and $c_2(E)=c$ for any integers $r\ge 2$ and $c>0$ (Theorem  \ref{f11}).

In section \ref{Sh},
for each integer $r\ge 2$,  we prove the existence of  irreducible rank $r$ vector bundles on $X$ with trivial determinant, $c_2=1$ and whose restriction to each fiber is semistable, i.e. with no jumps (Theorem\thinspace\ref{h10}).

In section \ref{Sloc}, starting with elements of  $\mathcal M_{r,n}$, 
we construct families of torsion free sheaves living in the closure of 
$\mathcal M_{r,n+1}$. Namely, 
 assuming that  $E$ is a simple (e.g. stable) vector bundle over $X$ of rank $r$,  
we show that the kernels of the surjections $E\rightarrow p$ with $p\in X$ 
form  an $(r+1)$-dimensional family of pairwise non-isomorphic rank $r$ torsion free sheaves on $X$ 
(Proposition\thinspace\ref{cod2}) and prove that such torsion free sheaves are limits of vector bundles 
(Theorem\thinspace\ref{boundary}).
 
\section{Hopf surfaces and stability}\label{hf}
 This section contains basic definitions.
 A {\it classical Hopf surface} is an elliptically  fibered surface $X\to \mathbb P^1$ 
which is diffeomorphic to $S^1\times S^3$. More precisely,
a  Hopf surface is  a quotient of the punctured plane $\mathbb C^2\setminus \{0\}$ 
by the action of an infinite cyclic group generated by a contraction of $(\mathbb C^2,0)$.
When the contraction is given by an action of the form 
 $(z_1,  z_2)  \mapsto  (\mu z_1,  \mu z_2)$ the quotient is called a classical Hopf surface.  
Hence, it  admits a natural holomorphic elliptic fibration 
\[ \begin{array}{rcl}
\pi\colon X & \rightarrow & \mathbb{P}^1 \\
(z_1,  z_2) & \mapsto & [ z_1 :  z_2]
\end{array} \]
with fiber the elliptic curve $T = \mathbb{C}^\ast/\mu$.

Given a classical Hopf surface $X$, the Picard group of all holomorphic line bundles on $X$ is given by 
$$\Pic(X) = \frac{H^1(X, \mathcal O)}{H^1(X, \mathbb Z)}= \frac{ \mathbb C}{\mathbb Z}= \mathbb C^*,$$
and since $H^2(X, \mathcal O)=0$, 
 every line bundle on $X$ has zero first Chern class, hence 
$\Pic^0(X) = \Pic(X)=  \mathbb C^*.$  
The relative Jacobian of a classical Hopf surface $X \stackrel{\pi}{\rightarrow} \mathbb{P}^1$ 
is 
\[ J(X) \simeq \mathbb{P}^1 \times T^\ast \stackrel{p_1}{\rightarrow} \mathbb{P}^1, \]
where $T^\ast  \ce \Pic^0(T) $ is the dual elliptic curve.
 Holomorphic rank $2$ vector bundles on elliptically fibered Hopf surfaces were discussed in \cite{BMo,Mo1,Mo2}. 
 Recall that a bundle on $X$ is 
referred to as {\it regular} when it is regular on every fiber of $X \mapsto \mathbb P^1$, that is,  when its
 automorphism group has minimal dimension,  on every fiber. A bundles is called {\it irregular} if it is not 
 regular. Hypersurfaces of the moduli of rank $2$ bundles on the Hopf surface
 formed by irregular bundles  were described 
 in \cite{BG}.

  A metric on $X$ is called {\it Gauduchon} if its associated $(1,1)$-form $\omega$ satisfies $\partial \bar\partial \omega = 0$.
 The degree of a line bundle $L$ on $X$ is defined as 
$$ \deg(L) = \int_X F\wedge \omega $$
where $F$ is the curvature of a Hermitian connection on $L$ compatible with $\bar \partial_L$.
For a coherent sheaf $E$, de {\it degree} is defined as 
$$\deg(E) = \deg\det(E)$$
and the {\it slope} by 
$$\mu(E) = \frac{\deg(E)}{\rank(E)}.$$
Up to a scalar there is a unique Gauduchon metric on $X$  (we fix it)
 so that each $f^\ast({\mathcal O}_{\mathbb P^1}(t))$ has degree $t$. 
 Stability is then defined as usual, that is, 
 $E$ is {\it stable} any subsheaf $F \hookrightarrow E$ satisfies
 $\mu F < \mu E$. \\

  Atiyah \cite{a} proved that an $SL(2,\mathbb C)$-bundle over an elliptic curve $T$
  is of one of three possible types: 
  \begin{itemize}
  \item[$\iota)$] $L_0\oplus L_0^*$, with $L_0\in \Pic^0(T)$.
  \item[$\iota\iota)$] Non-trivial extensions $0\rightarrow L_0\rightarrow E \rightarrow L_0\to 0$, with $L_0^2=\mathcal O$.
  \item[$\iota\iota\iota)$] $L\oplus L^*$, with $L\in \Pic^{k}(T)$,  $k<0$.
  \end{itemize}

  Given an $SL(2,\mathbb C)$-bundle $E$ over the Hopf surface $\pi\colon X\rightarrow \mathbb P^1$, 
  its restriction to the generic fiber of $\pi$ is of type $\iota)$ or $\iota\iota)$. Furthermore, 
  the restriction is of type $\iota\iota\iota)$ at finitely many fibers \cite[Prop\thinspace3.2.2]{BH}; 
  named the { jumping fibers} of $E$. 
Observe that in the rank $2$ case $T$ is a jumping fiber of $E$ precisely when $E_{|T}$ is not 
semistable. Therefore, for rank any rank $r$, we define jumps as follows.

\begin{definition} We say that the vector bundle $E$ has a {\it jump} at $T$   
if  the restriction $E_{|T}$ is not 
semistable. If $E_{|T}$ is semistable at all fibers $T$ of $\pi$ it  is said to have no jumps.
\end{definition}

Given a rank 2 bundle $E$ on $X$, let $V$ be the universal Poincaré line bundle on
$X \times \mathbb C^*$ and consider the first derived image $R^1\pi_*(E \otimes V)$ where
$\pi$  the projection $\pi\colon X\times \mathbb C^* \to \mathbb P^1 \times \mathbb C^*$. 
The sheaf $R^1\pi_*(E \otimes V)$ is supported on a divisor on $ \mathbb P^1\times \mathbb C^*$, 
which descends to a divisor 
$$D \subset \mathbb P^1\times \mathbb P^1 = \pi(X) \times \Pic^0(T)/\pm1.$$ When $c_2(E)=n$, the divisor
 $G(E) \ce D $ belongs to the linear system $|\mathcal O(n,1)|$ over $\mathbb P^1\times \mathbb P^1$
 and is called the {\it graph} of $E$.
 
 Note that here there is a 2-sheeted map $\Pic^0(T) \rightarrow \mathbb P^1$, whose 
 branch points are the half-periods. 
Such divisor  takes the form $D=\sum_{i=1}^k (x_i \times\mathbb P^1) + Gr(F)$ where $x_i$ 
are the location of the jumping lines of $X$ and $F\colon \mathbb P^1\rightarrow \mathbb P^1$ is a
rational map of degree $n-k$.  \cite{BH} also show that if this divisor includes the 
graph of a non-constant rational map, then the bundle $E$ is stable. In particular, 
this implies that all rank $2$ bundles on $X$  with $c_2>0$  and without jumps are stable. 

\section{Filtrable bundles with jumps}\label{Sf}

The main goal of this section is to prove Th.\thinspace \ref{f11}, showing the existence of 
filtrable stable vector bundles on the 
Hopf surface $X$ with trivial determinant and having any rank $r\geq 2$ 
and any prescribed value of positive second Chern class. 
In this section, we also prove Th.\thinspace \ref{f00} that shows that  filtrable bundles must have jumps.

We first recall the  definition of  filtrability.
 
\begin{definition} Let $Y$ be a smooth and connected compact complex manifold. 
A rank $r>1$ torsion free sheaf $E$ on $Y$ is said to be {\em filtrable}
 if there are sheaves $$0 =E_0\subset E_1\subset \cdots E_i\subset \cdots \subset E_r=E$$
  such that each $E_i$ has rank $i$ and $E_i/E_{i-1}$ has no torsion for $i=2,\dots ,r$. 
  \end{definition}
  
  The latter condition may be omitted (Rem.\thinspace \ref{f1}). We say that a rank $1$ torsion free sheaf is filtrable. If $Y$ is projective, then every torsion free sheaf on $Y$ is filtrable (Rem.\thinspace \ref{f2}). We are
  mostly interested in the case when $E$ is locally free \cite[Def.\thinspace 4.6]{brin}, but even
   then, it would be too restrictive to assume that $E_r/E_{r-1}$  are line bundles (Rem.\thinspace \ref{f3}, Lem.\thinspace \ref{f4}).

The following well-known construction is related to \cite[Def.\thinspace 1.1.5]{HL}.
\begin{remark}\label{f1}
Let $E$ be a rank $r>1$ torsion free sheaf on $Y$ with a filtration $F_\ast = \{0 =F_0\subset F_1\subset \cdots F_i\subset \cdots \subset F_r=E\}$ with $\mathrm{rank}(F_i) = i$ for all $i$. We want to obtain another  filtration $0 =E_0\subset E_1\subset \cdots E_i\subset \cdots \subset E_r=E$ such that each $E_i$ has rank $i$ and $E_i/E_{i-1}$ has no torsion for $i=2,\dots ,r$. Since $E$ has no torsion, each $F_i$ has no torsion. Call $v\colon E\to E/F_{r-1}$ the quotient map. Let $Q\subset E/F_{r-1}$ be the torsion of $E/F_{r-1}$. Hence $Q=0$ if and only if $E/F_{r-1}$is torsion free. Set $E_{r-1} =v^{-1}(Q)$. 
 Note that $E_{r-1}$ is a rank $r-1$ subsheaf of $E$ containing $F_{r-1}$ and hence containing $F_{r-2}$. If $r=2$, then $0\subset E_1\subset E$ is the promised filtration. Assume $r>2$.
Then we continue with $F_0\subset \cdots F_{r-2}\subset E_{r-1}$. The filtration $0 =E_0\subset E_1\subset \cdots E_i\subset \cdots \subset E_r=E$ is called  the {\em saturation} of $F_\ast$. 
\end{remark}

\begin{remark}\label{f01}
If $E$ has rank $2$, then $E$ is filtrable if and only if it is not irreducible, hence it has 
a proper subsheaf of rank $1$. In rank $2$ a bundle is filtrable if it is not stable with 
respect to any Gauduchon metric.
\end{remark}

\begin{remark}\label{f2}
We recall that
on a projective manifold every torsion free sheaf is filtrable \cite[p.\thinspace 91]{brin}.
\end{remark}

\begin{remark}\label{f3}
Let $E$ be a rank $r>1$ vector bundle with a filtration $0 =E_0\subset E_1\subset \cdots E_i\subset \cdots \subset E_r=E$ such that each $E_i$ has rank $i$ and $E_i/E_{i-1}$ is a line bundle. We say that $E$ is an iterated filtration by line bundles. This is a very restrictive assumption. For instance, let $\PP^n$, $n>1$, be a projective space,
then we have $h^1(\PP^n,L)=0$ for all line bundles $L$ on $\PP^n$. Hence every short exact sequence 
with  a line bundle on the right and a direct sum of line bundles on the left splits. Thus, by induction on $r$ we get that each iterated filtration by line bundles is isomorphic to a direct sum of line bundles. 
Therefore, neither $T\PP^n$ nor stable bundles are iterated filtrations by line bundles.
See Remark \ref{f4} for the main reason to avoid the restrictive definition with line bundles on Hopf surface.
\end{remark}

\begin{lemma}\label{f4}
Let $X$ be a Hopf surface. Let $E$ be a rank $r>1$ vector bundle on $X$ which is an iterated filtration of line bundles,
then $c_2(E)=0$.
\end{lemma}
\begin{proof}
To verify this fact, consider an exact sequence of vector bundles on $X$
\begin{equation}\label{eqf1}
0 \to A\to B\to C\to 0.
\end{equation}
  The Chern classes of the middle term of a short exact sequence are computed in terms of the Chern classes of the left and right terms. Clearly, $\det(B)\cong \det(A)\otimes \det(C)$. Since $H^2(X,\ZZ)=0$, the intersection form $H^2(X,\ZZ)\times H^2(X,\ZZ)\to H^4(X,\ZZ)\cong \ZZ$ is the zero-form. Therefore, $c_1(A)\cdot c_1(C)=0$ and
   $c_2(B)=c_2(A)+c_2(C)$. Since $c_2(L)=0$ for every line bundle $L$, we get $c_2(E)=0$ by induction on $r$.
\end{proof}

\begin{lemma}\label{f5}
Let $Y$ be a smooth compact surface. Let $E$ be a rank $r>1$ vector bundle $E$ equipped with a filtration $0 =E_0\subset E_1\subset \cdots E_i\subset \cdots \subset E_r=E$ such that each $E_i$ has rank $i$ and $E_i/E_{i-1}$ has no torsion for $i=2,\dots ,r$. Then each $E_i$, $1\le i\le r-1$, is locally free.
\end{lemma}

\begin{proof}
The sheaf $E/E_{r-1}$ has on torsion by assumption. If $1\le i\le r-2$ the sheaf $E/E_i$ has no torsion, because it is an iterated extension of torsion free sheaves. Since $Y$ is a smooth surface, $E_i$ is a subsheaf of the locally free sheaf $F$ and $E/E_i$ has no torsion, $E_i$ is locally free.
\end{proof}

\begin{remark}\label{le2}
Let $X$  be a non-algebraic compact complex surface. 
 \cite[Th.\thinspace 4.17]{brin} shows that the existence of a holomorphic rank $r$  vector bundle $F$ on $X$ implies $\Delta(F)\ge 0$;
where  $\displaystyle \Delta(F)  = \frac{1}{r}\left( c_2(F)- \frac{r-1}{2r}c_1^2(F)\right)$ is the 
discriminant of $F$.
\end{remark}

\begin{theorem}\label{f00}
Let $E$ be a filtrable vector bundle on a classical Hopf surface $X$  such that $c_2(E)>0$. Then there 
exists a fiber $J$ of the fibration $u\colon X\to \PP^1$ such that $F_{|J}$ is not semistable.
\end{theorem}

Recall (e.g. \cite{FMW}) that for rank $r>2$ restriction to a fiber of the elliptic fibration being not semistable is the right notion that extends the notion of ``having a jump'' for the case $r=2$.
Note that in Theorem \ref{f00} we do not assume that $E$ is stable. Hence it is more general, but we only use it here
for stable bundles. We first prove a lemma.

\begin{lemma}\label{f6}
Let $J$ be a fiber of $u$. Consider an exact sequence \eqref{eqf1} with $A$ and $B$ locally free and $C$ torsion free. If $B_{|J}$ is semistable, then
$A_{|J}$ is semistable.
\end{lemma}

\begin{proof}
Since $C$ is torsion free, it is locally free outside finitely many points of $X$, say outside a finite set $S$. Hence restricting \eqref{eqf1} to $J$ we get an exact sequence
\begin{equation}\label{eqf2}
A_{|J} \stackrel{v}{\to} B_{|J} \to C_{|J}\to 0
\end{equation}
in which $u$ is injective outside finitely many points, $S\cap J$. Since $A_{|J}$ and $B_{|J}$ are vector bundles on $J$, $v$ is injective. Assume that $A_{|J}$ is not semistable. Recall that $J$ is an elliptic curve and that for any vector bundle $E$ on $X$ the vector bundle $E_{|J}$ has degree $0$. Since $J$ is an elliptic curve, $\deg(A_{|J})=0$ and $A_{|J}$ is not semistable, there is an indecomposable factor $G$ of $A_{|J}$ with $\deg(G)>0$. Since $\deg(B_{|J})=0$, $v(G)\cong G$ has degree $>0$ and $v(G)$ a subsheaf of $B_{|J}$, $B_{|J}$ is not semistable.
\end{proof}

\begin{proof}[Proof of Theorem \ref{f00}:]
Take a filtration $0 =E_0\subset E_1\subset \cdots E_i\subset \cdots \subset E_r=E$ such that each $E_i$ has rank $i$ and $E_i/E_{i-1}$ has no torsion for $i=2,\dots ,r$. Lemma \ref{f5} gives that $E_{r-1}$ is locally free. The sheaf $E/E_{r-1}$ is torsion free and it has rank $1$. Thus $E/R_{r-1}\cong \Ii_Z\otimes L$ for some line bundle on $X$ and $Z\subset X$ is a zero-dimensional scheme. We have an exact sequence
\begin{equation}\label{eqf3}
0\to E_{r-1}\to E\to \Ii_Z\otimes L\to 0.
\end{equation}
For any line bundle $R$ on the Hopf surface we have $c_2(R)=0$. We have $c_2(\Ii_Z\otimes R)=\deg(Z)$ (Rem.\thinspace \ref{h3}). Hence as in Remark \ref{f4} we get
$c_2(E) =c_2(E_{r-1}) +\deg(Z)$.

Let $J$ be a fiber of $u$. If $E_{r-1|J}$ is not semistable, then $E_{|J}$ is not semistable (Lem.\thinspace \ref{f6}). Hence we may assume that the restriction of $E_{r-1}$ to all fibers are semistable. If $E_{r-1}$ is a line bundle, i.e. if $r=2$,  then $c_2(E_{r-1})=0$. If $r>2$ by the induction on the rank we get $c_2(E_{r-1})\le 0$ and hence $c_2(E_{r-1})=0$ by Remark \ref{le2}. Since $c_2(E)>0$, $Z\ne \emptyset$. Fix $p\in Z_{\red}$ and let $J$ be the fiber of $u$ containing $p$.
We want to prove that $E_{|J}$ is not semistable. Note that $\deg(L_{|J})=0$. Since the tensor product is a right exact functor and restricting to $J$ is just the tensor product with $\Oo_J$ seen as quotient of the sheaf $\Oo_X$, \eqref{eqf3} induces a surjection $w\colon E_{|J} \to (\Ii_Z\otimes L)_{|J}$.
Since $p\in Z$, $(\Ii_Z\otimes L)_{|J}$ has torsion. Hence the surjectivity of $w$ implies the existence of a factor of $E_{|J}$ of degree $>0$. i.e. $E_{|J}$ is not semistable.
\end{proof}

\begin{remark}\label{f7}
We collect some facts about rank $2$ bundles on  a classical Hopf surface proved in 
\cite[Prop.\thinspace 3.3.3, Th.\thinspace 5.2.2]{BH}. 
Recall that for rank $2$ irreducibility is equivalent to non-filtrability (Rem.\thinspace \ref{f01}).
\begin{itemize}
\item[(a)] For all $c_2>0$ there exists an irreducible (hence stable for any Gauduchon metric) 
 rank $2$  vector bundles and the set of all such bundles is a connected manifold of dimension $4c_2$
 .

\item[(b)] For all $c_2\ge 2$ there are irreducible rank $2$ bundles with jumps and, counting multiplicity, 
these may have any number of jumps between $1$ and $c_2$.

\item[(c)] For all $c_2>0$ there are stable filtrable rank $2$ bundles. 

\item[(d)] A stable rank $2$ vector bundle with $c_2=1$ is filtrable if and only if it has a jump (which must be unique). The description of rank $2$ bundles with $c_2=1$ and a jump gives that they have a unique jump (even the non-stable ones).
\end{itemize}
\end{remark}

\begin{remark}\label{f8}
Let $E$ be a rank $r$ vector bundle such that there is a fiber $J$ of $u$ 
where the degree $0$ rank $r$ vector bundle $E_{|J}$ is not semistable. Since
$J$ is an elliptic curve and $\deg(E_{|J}) =0$, it means that $E_{|J}$ has a factor of negative degree. Since $E$ is locally free, we have $(E^{\vee})_{|J} \cong (E_{|J})^\vee$. Thus $E^\vee_{|J}$ has a factor $G$ of positive degree. See $G$ as torsion sheaf on $X$. Let $w: E^{\vee} \to G$ denote the composition of the restriction map with the surjection $E^{\vee}_{|J} \to G$. Set $A:= \mathrm{ker}(w)$. Since $G$ is torsion and $A\subset E^\vee$, $A$ is a torsion free sheaf of rank $r$. We get an exact sequence
\begin{equation}\label{eqf4}
0 \to A \to E^\vee \stackrel{w}{\to} G\to 0.
\end{equation}
Since $E^{\vee}$ is locally free and $G$ is a locally free sheaf of a Cartier divisor of $X$, $A$ has depth at least $2$ at each point of $X$. Since $X$ is a smooth surface, $A$ is locally free. Apply the functor $\Hom_{\Oo_X}(-,\Oo_X)$ to \eqref{eqf4}. We have $\Hom_{\Oo_X}(G,\Oo_X)=0$. We identify $E^{\vee\vee}$
with $E$. The higher ext-groups of the bundles $A$ and $E^\vee$ are zero. Since $G$ is a locally free $\Oo_J$-sheaf and $J$ is a Cartier divisor, the sheaf
$B:= \Ext^1_{\Oo_X}(G,\Oo_X)$ is a locally free $\Oo_J$-sheaf dual to $G$ and hence of positive degree as an $\Oo_J$-vector bundle. Thus, applying the functor $\Hom_{\Oo_X}(-,\Oo_X)$ to \eqref{eqf4} we get the exact sequence
\begin{equation}\label{eqf5}
0\to A^\vee \to E\to B\to 0
\end{equation}
with $c_2(B)>0$. Since $c_2(E)=1$, $c_2(A^\vee) \le 0$. By Remark \ref{h3} $c_2(A^{\vee})=0$ and hence $c_2(B)=1$. Proposition \ref{h6} gives that $A^\vee$ is filtrable and that the restriction of $A^\vee$ to each fiber of $u$ is semistable. 
Thus, \eqref{eqf5} shows that the restriction of $E$ to all other fibers of $u$ is semistable. Since $c_2(B)=1$, we see that there is a unique irreducible factor of $E_{|J}$ with negative degree and that such factor has degree $1$. Since $\deg(E_{|J})=0$, $E_{|J}$ has a unique irreducible factor of positive degree and this degree is $1$.
\end{remark}
\begin{theorem}\label{f9}
Let $E$ be a rank $r$ vector bundle on $X$ with $c_2(E)=1$. Assume the existence of a fiber $J$ of $X$ such that $E_{|J}$ is not semistable. Then $E$ is filtrable, $J$ is the unique fiber such that $E_{|J}$ is not semistable and $E_{|J}$ as a unique irreducible factor of degree $1$, a unique irreducible factor of degree $-1$, while the other factors, if any, have degree $0$.
\end{theorem}

\begin{proof}
If $r=2$ the result is true by Remarks \ref{f01} and \ref{f7}. We will not use this observation in the proof.
Suppose you have a subsheaf $F\subset E$  of rank $r$ and $c_2(F) \le 0$. By Remark \ref{h5} $F$ is locally free and $c_2(F)=0$. By Proposition \ref{h6} $F$ is filtrable, and remark \ref{f1} gives that $E$ is filtrable. Thus, to conclude the proof it is sufficient to prove the existence of the subsheaf $F$. Take $A^\vee$ from Remark \ref{f8}.
\end{proof}

\begin{remark}\label{f10}
Let $Y$ be a smooth and compact complex surface. Let $A$, $B$ be line bundles on $Y$. Fix a positive integer $c$ and take $c$ distinct points $p_1,\dots ,p_c$ of $X$. Set $S:= \{p_1,\dots,p_c\}$ and $S_i:= S\setminus \{p_i\}$. Assume $h^0(Y,\Ii_{S_i}\otimes B\otimes A^\vee\otimes \omega_Y)=0$ for all $i=1,\dots,c$.
This is the so-called Cayley--Bacharach condition. Then the general sheaf $E$ fitting into an exact sequence
\begin{equation}\label{eqf6}
0\to A\to E\to \Ii_S\otimes B\to 0
\end{equation}
is locally free (\cite[Th.\thinspace 1.4]{c1}). If  $h^0(Y,B\otimes A^\vee\otimes \omega_Y)=0$, then the Ext-group of \eqref{eqf6} is a vector space of dimension $h^1(Y,A\otimes B^\vee)+c$ (\cite[Th.\thinspace 1.4]{c1}).
\end{remark}

\begin{lemma}\label{f12}
Assume that we normalized $\deg_g$ so that $\deg_g(u^\ast(\Oo_{\PP^1}(1)))=1$. Take $L\in \mathrm{Pic}(X)$ such that $0\le \deg_g(L)<1$. Then $h^0(L)>0$ if and only if $L\cong \Oo_X$. 
\end{lemma}

\begin{proof}
Recall that the unique irreducible curves $C\subset X$ are the fibers of the elliptic fibration. Take $s\in H^0(L)$ such that $s\ne 0$ and set $D:= \{s=0\}$. We have $L\cong \Oo_X(D)$. If $D=\emptyset$, then  $L\cong \Oo_X$. We have $h^0(\Oo_X)=1$ and $\deg_g(\Oo_X)=0$. If $D\ne \emptyset$, then $D$ is a union of $s\ge 1$ fibers of $u$. Thus $\deg(L) =s\ge 1$, a contradiction.
\end{proof}
\begin{remark}\label{f13}
Let $Y$ be a smooth  compact complex surface equipped with a Gauduchon metric $g$. Let $j\colon A\to B$ be a morphism between vector bundles of the same rank, $m$, such that $j(A)$ has rank $m$. Then $j$ is injective. The map $j$ is an isomorphism if and only if the determinantal map $\det(j)\colon \det(A)\to \det(B)$ is an isomorphism. By the definition of $\deg_g$ we have $\deg_g(A)=\deg_g(\det(A))$ and $\deg_g(B)=\deg_g(B)$. We have $\deg_g(A)\le \deg_g(\det(B))$ with equality if and only if $j$ is an isomorphism.
\end{remark}
\begin{lemma}\label{f14}
Let $X$ be a classical Hopf surface. Let $j\colon A\to B$ be a morphism between vector bundles of the same rank, $m$, such that $j(A)$ has rank $m$.  If $\deg_g(A)>\deg_g(B)-1$, then $j$ is an isomorphism.
\end{lemma}

\begin{proof}
By Lemma\thinspace \ref{f12} $\det(j)$ is an isomorphism, therefore by Remark\thinspace \ref{f13},  $j$ is also an isomorphism .
\end{proof}

\begin{lemma}\label{f15}
Let $j\colon A\to B$ be an injective map of vector bundles, with $\mathrm{a}\ce\rank(A)<\mathrm{b}\ce \rank(B)$. Assume that $A$ is filtrable by line bundles and fix a filtration $0=A_0\subset \cdots \subset A_a=A$ with $A_{i+1}/A_i$  line bundles. Assume that $B$ is a direct sum of $\mathrm{b}$ line bundles. Then $B$ has a rank $\mathrm{a}$ factor $B'$  which contains a subsheaf isomorphic to $A$.
\end{lemma}

\begin{proof}
Write $B=\oplus_{j=1}^{b} B_j$ with $B_j$ a line bundle. 

We will prove that for all $i=1,\dots,a$ there is a rank $i$ factor $B'_i$ of $B$ containing a subsheaf isomorphic to $A_i$. The case $i=1$ is equivalent to the existence of $j\in \{1,\dots,b\}$ and a non-zero map $A_1\to B_j$. The existence of $j$ follows because $h^0(\Hom(A_1,B)) =\sum_{j=1}^{b} h^0(\Hom(A_1,A_j))$.

Now take  $2\le i\le a$ and assume the existence of a rank $i-1$ factor $B'_{i-1}$ of $B$ containing a sheaf isomorphic to $A_{i-1}$. Write $B = B'_{i-1}\oplus B''$ with $B''$ a factor of $B$. Permuting the integer $1,\dots ,b$ we may assume $B'_{i-1} =\oplus_{j=1}^{i-1}$ and $B''=\oplus_{j=i}^{b} B_j$. Let
$\pi: B\to B''$ denote the projection. Call $\alpha$ the isomorphism between $A_{i-1}$ and a subsheaf of $B_{i-1}$. Since $j$ is injective $j(A_i)$ has rank $i$ and hence it is not contained in $B_{i-1}$. Thus $\pi_{|A_i}\ne 0$. Since $\alpha(A_{i-1})$ is contained in the kernel of the surjection $\pi\colon B\to B''$,
we get that the line bundle $A_i/A_{i-1}$ is isomorphic to a subsheaf of a factor of $B''$, say $B_i$. 
Since $j(A_i)\subseteq B$ and  $\alpha(A_{i-1})$ is contained in the kernel of $\pi$, $A_i$ is isomorphic to a subsheaf $\oplus_{j=1}^{i} B_j$ of $B$.
When $i=a$ we get the lemma.\end{proof}

\begin{theorem}\label{f11}
Fix integers $r\ge 2$ and $c>0$. Then there exists a stable and filtrable rank $r$ vector bundle on $X$ such that $\det(E)\cong \Oo_X$ and $c_2(E)=c$.
\end{theorem}

\begin{proof}[Proof of Theorem \ref{f11}:]
Fix $r-1$ pairwise not isomorphic line bundles $A_1,\dots,A_{r-1}$ on $X$ such that $\deg_g(A_i) =0$ for all $i$. Fix points $p_1,\dots ,p_c\in X$ such $u(p_i)\ne u(p_j)$ for all $i\ne j$. Set $S:= \{p_1,\dots,p_c\}$. Take a line bundle $L$ such that $\deg_g(L) =1/2$ (\cite[(2) at p.\thinspace 140]{brin}.  Let $F$ be the general sheaf fitting in an exact sequence
\begin{equation}\label{eqf7}
0\to \oplus_{i=1}^{r-1}A_i\to F\to \Ii_S\otimes L\to 0 .
\end{equation}
Since $\omega_X\cong u^\ast(\Oo_{\PP^1}(-2))$, we have $\deg_g(\omega_X\otimes L\otimes A_i^\vee)=-3/2$ for all $i$. Hence $h^0(\omega_X\times L\otimes A_i^\vee)=0$ for all $i$. 
Remark \ref{f11} gives that $F$ is locally free. Note that $\det(F)\cong L\otimes A_1\otimes \cdots\otimes A_{r-1}$. By Remark \ref{h8} we have $c_2(F) =\#S =c$. 
By \eqref{eqf7} $F$ is filtrable. Since $\mathrm{Pic}(X)\cong \CC^*$ as an abelian group, there is $R\in \mathrm{Pic}(X)$ such that $R^{\otimes r} \cong\det(F)^\vee$. 
Note that $\deg(R) = -1/2r$. We have $\deg_g(G) =-1/2r$ and hence $0<\deg_g(L\otimes R)<1$. Set
$E:= F\otimes R$. We have $\det(E)\cong \Oo_X$, hence $\deg_g(E)=0$. Since $F$ is filtrable, $E$ is filtrable. By Remark \ref{h8} we have $c_2(E)=c$. Thus, to conclude the proof of Theorem \ref{f11} it is sufficient to prove that $E$ is stable. Assume that $E$ is not stable. Let $G$ be a minimal rank saturated subsheaf of $E$
with $\deg_g(G)\ge 0$. Since $G$ is saturated in $E$, $E/G$ has no torsion and hence $G$ is locally free. 
Call $w\colon E\to \Ii_S\otimes (L\otimes R)$ the surjection obtained twisting by $R$ the exact sequence \eqref{eqf7}. The minimality of $G$ gives that $G$ is stable. Since $\deg_g(A_i\otimes R)<0$ , $F$ is not contained in $\otimes_{i=1}^{r-1} A_i\otimes R$. Hence $w(G)$ is a non-zero subsheaf
of $\Ii_Z\otimes L\otimes R$. Thus $w(G) =\Ii_W\otimes M$ for some line bundle $M$. Since $G$ is stable and $\deg_g(G)\ge 0$, ether $G=w(G)$ or $\deg_g(w(G)) >0$. If $G=w(G)$ we have $\deg_g(w(G)) \ge 0$. We have the exact sequence
\begin{equation}\label{eqf8}
0\to G\to E\to E/G\to 0
\end{equation}
with $E/G$ torsion free.

\quad 
Observe that  the Ext-group of \eqref{eqf7} has dimension $(r-1)(c+1)$  by Lemma \ref{f12}.
\begin{itemize}
\item[(a)] Assume $M\ne L\otimes R$. Taking the double dual of the inclusion $\Ii_W\otimes M\subseteq \Ii_Z\otimes L\otimes R$, we get a non-zero map $\alpha\colon M\to L\otimes R$. Since $0\le \deg_g(L\otimes R\otimes M^\vee) <1$ and $M\ne L\otimes R$, the existence of $\alpha$ contradicts Lemma \ref{f12}.

\item[(b)] Assume $M=L\otimes R$ and $W\ne Z$. Since $w(G) \subset L\otimes R$, we get $W\supsetneq Z$ and hence $c_2(w(G)) =\deg(W)>c$.
Since $c_2$ is additive in exact sequence of torsion free sheaves on $X$ (Rem.\thinspace \ref{h8}), we get $c_2(G)\ge c_2(w(W)) >c$. From \eqref{eqf8} we get $c_2(E)>c$, a contradiction.

\item[(c)] By steps (a) and (b) we may assume $w(W) =\Ii_Z\otimes L\otimes R$.  Let $s$ be the rank of $G$. 
Set $W\ce F\cap \mathrm{ker}(w)$. If $W=0$, then $F=\Ii_Z\otimes L\otimes R$ is not locally free, a contradiction. 
Therefore, $W\ne 0$, i.e. $s\ge 2$. Thus we proved the (known by \cite{BH}) case $r=2$. 
Now, we assume $r>2$ and that the theorem is true for all ranks $<r$. Set $U\ce \oplus_{i=1}^{r-1} A_i\otimes R$. Note that $U$ is a polystable with slope $\deg(A_i\otimes R) = -1/2r$.
Since $G/w(G) =\Ii_Z\otimes L\otimes R$ and $W$ is a subsheaf of the locally free sheaf $U$, $W$ is locally free. We have $c_2(W)\ge 0$ and $c_2(W)=0$ if and only if it is filtrable by line bundles (Prop.\thinspace\ref{h6}). Since $c_2(w(W)) =c$, we get $c_2(W)=0$ and hence $W$ is filtrable by line bundles.
Let $\underline{\epsilon}(\epsilon_1,\dots ,\epsilon_{r-1})$ denote the extension inducing $E$ with $\epsilon_i$ a general extension of $\Ii_Z\otimes L\otimes R$ by $A_i\otimes R$. Since $U$ is a direct sum of line bundles and $W$ is filtrable by line bundles, there is a factor $U_{s-1}$ of $U$ containing a sheaf $W'$ isomorphic to $W$ (Lemma \ref{f15}). First assume $W' \ne U_{s-1}$. By Remark \ref{f13} we have $\deg_g(W) =\deg_g(W') \le -1 +\deg(U_{s-1}) <-1$.
Hence $\deg_g(G)\le \deg_g(L\otimes R) < -1+1/2$, a contradiction. Thus $W' =U_{s-1}$. Since $W\cong W'$ and $W$ is a saturated subsheaf of $U$, $U/W$ is isomorphic to a direct sum of $r-s$ of the rank $1$ of $U$. We have $E/G \cong U/W$. Thus $E$ has a quotient isomorphic to $r-s$  rank $1$ factors of $i$, say
to $\oplus_{i=1}^{r-s} A_i\otimes R$. Since the line bundles $A_i\otimes R$ are pairwise not isomorphic, we get $\epsilon_i=0$ for all $i\le r-s$, contradicting the generality of $\underline{\epsilon}$.
\end{itemize}\end{proof}

\section{Bundles with $c_2=1$ without jumps}\label{Sh}

Here we consider a classical Hopf surface $X$, $u\colon  X\to \PP^1$ the elliptic fibration, 
$T$ a fiber of $u$. We fix a Gauduchon metric $g$ on $X$, normalized so that $\deg_g(u^\ast(\Oo_{\PP^1}(1)) =1$. Stability will be slope stability with respect to $\deg_g$. In this section we consider vector bundles and torsion free sheaves on $X$, 
showing in Th.\thinspace \ref{h6} that any torsion free sheaf $E$ on $X$ such that $c_2(E)=0$  is filtrable,  locally free,
 and  is not stable for any Gauduchon metric $g$ on $X$. 
We also describe precisely the case when  $E$ is semistable.
We conclude this section by showing in Th.\thinspace \ref{h10} that, for each integer $r\ge 2$,
there exists an irreducible rank $r$ vector bundle on $X$ with trivial determinant, $c_2=1$  with no jumps.

For all $r\ge 2$ and $c>0$ let $\Mm_{r,c}$ denote the moduli space of all rank $r$ stable vector bundles $E$ on $X$ with $\det(F)\cong \Oo_X$ and $c_2(X)=c$. Let $\widetilde{\Mm}_{r,c}$ denote the moduli space of all rank $r$ stable torsion free sheaf $F$ on $X$ with $\det (F)\cong \Oo_X$ and $c_2(F)=c$. Let $\widehat{\Mm}_{r,c}$ denote the closure of $\Mm_{r,c}$ inside $\widetilde{\Mm}_{r,c}$.

\begin{remark}\label{h1}
For any smooth non-K\"{a}hler complex surface and any Gauduchon metric $g$ on $X$ the map 
$\deg_g\colon \mathrm{Pic}^0(X) \to \RR$ is a surjective homomorphism of real Lie groups 
(\cite{Buchdahl}, \cite[p.\thinspace 140]{brin}). 
If $X$ is a classical Hopf surface, then $\mathrm{Pic}^0(X) = \mathrm{Pic}^0(X)\cong \CC^*$. 
Since $\deg_g\colon \CC^*\to \RR$ is a surjection of real Lie groups, $\deg_g^{-1}(0)$ is isomorphic (as a topological group) to $S^1$ seen as the rotation group of $\RR^2$.
Fix an integer $r\ge 2$. There are $r$ rotations by angles $\omega_1,\dots ,\omega_r$ such that $\omega _1+\cdots +\omega_r\equiv 0\pmod{2\pi}$, but for no
$s\in \{1,\dots,r-1\}$ and $\{i_1,\dots,i_s\}\subset \{1,\dots ,s\}$ and $(m_1,\dots ,m_s)\in \ZZ^s\setminus \{(0,\dots ,0)\}$ the real number $m_1\omega_{i_1}+\cdots +m_s\omega_{i_s}\not\equiv 0\pmod{2\pi}$. We may achieve these conditions by taking $(\omega_1,\omega_{r-1})$ general in $(S^1)^{r-1}$ and then taking $\omega _r=-\omega_1-\cdots -\omega_s$. 

Let $L_i$ be the element of $\deg_g^{-1}(0)$ associated to $\omega_i$. Set $F:= L_1\oplus \cdots\oplus L_r$. We have $c_2(F)=0$. Since $\deg_g(L_i)=0$ for all $i$, $F$ is semistable.
Since $\omega _1+\cdots +\omega_r\equiv 0\pmod{2\pi}$, $\det(F)\cong \Oo_X$. Note that $L_i$ and $L_j$ are not isomorphic for $i\ne j$.
Since $L_i\ne L_j$ for any $i,\ne j$, the sheaf $\End(F)$ is isomorphic to the direct sum of $r^2$ line bundles with degree $0$, exactly $r$ of them being trivial.
Since all irreducible curves contained in $X$ are the fibers of the elliptic fibration, each of them having degree $1$,
 \cite[pp.\thinspace 140--141]{brin} gives $h^0(\End(F)) =r$, each contribution coming from one of the $r$ trivial factors of 
 $\End(F)$. Since $\omega_X\cong u^\ast(\Oo_{\PP^1}(-2))$,
we have $h^0(\End(F)\otimes \omega_X)=0$. Duality gives $h^2(End(F))=0$. Hence the Kuranishi deformation space of $F$ is smooth
of dimension $h^1(\End(F))$. Since $H^2(X,\ZZ)=0$, the intersection form on $\mathrm{Pic}(X)$. Since $c_2(F)=0$, Riemann--Roch gives
$h^1(\End(F))=r-1$.
\end{remark}

\begin{remark}\label{h2}
Let $E$ be a rank $r$ simple vector bundle, e.g. a stable vector bundle, and $c_2(E)=c$. Since  $\omega_X\cong u^\ast(\Oo_{\PP^1}(-2))$,
we have $h^0(\End(F)\otimes \omega_X)=0$. Duality gives $h^2(\End(F))=0$. Hence the Kuranishi deformation space of $F$ is smooth
of dimension $4rc$. Thus each non-empty $\Mm_{r,c}$ is smooth, equidimensional and $\dim \Mm_{r,c} =2rc$.
\end{remark}

\begin{remark}\label{h3}
Let $F$ be a rank $r$ torsion free sheaf on $X$. Since $X$ is a smooth surface, $F^{\vee\vee}$ is locally free of rank $r$.
Since $F$ is torsion free, $F$ is in a natural way a subsheaf of $F^{\vee\vee}$. Set $G\ce F^{\vee\vee}/F$ and $z\ce h^0(G)$.
The support of $G$ is a finite set. Hence $\det(F)\cong \det(F^{\vee\vee})$ and $c_2(F) =c_2(F^{\vee\vee})+z$. Note that $z\ge 0$ and that $z=0$ if and only if $F=F^{\vee\vee}$, i.e. if and only if $F$ is locally free. 
By \cite[Prop.\thinspace 5.27]{brin} 
applied twice, 
$F$ is stable (resp. semistable) if and only if $F^{\vee\vee}$ is stable (resp. semistable).
\end{remark}

We recall the known emptiness or non-emptiness results on $\Mm_{r,c}$, $r>1$. It is known that $\Mm_{2,c}\ne \emptyset$ if and only if $c>0$.
For $r>2$,  $\dim \Mm_{r,c} =2rc$ and $\Mm_{r,c}=\emptyset$ if $c<0$ (this is also true for torsion-free non-locally free either by \cite[Th.1]{brFerrara} or applying Remark \ref{h3} to the locally free case, $F^{\vee\vee}$; the later implies the non-existence of torsion free, but not locally free stable bundles with $c_2=0$. We constructed in Remark \ref{h1} many semistable bundles with $c_2=0$. Then in \cite{brFerrara} he do only the case $r=2$ or $\dim X>2$.

\begin{remark}\label{h5}
Recall that on a classical Hopf there is no torsion free sheaf (any determinant) with $c_2<0$ (Remark \ref{le2} and that $c_1(F)^2=0$ for any $F$, because the intersection form of $X$ is zero. Let $E$ be any rank $r>1$ torsion free sheaf $E$ with $c_2(F)=0$ (any determinant). We get that $E$ is locally free.
\end{remark}

\begin{remark}\label{h5.1}
Fix $L\in \mathrm{Pic}(X)$. Recall that $\mathrm{Jac}(X) =X\times T^*$. As in \cite[Proof Cor.\thinspace 2.9]{T} or by \cite[3.1]{BH} with this identification
the line bundle $L\cap u^{-1}(x)\in T^*$ is the same for all $x\in \PP^1$.\end{remark}

\begin{remark}\label{h5.2}
By \cite[eq.\thinspace 3.1.1]{BH} the restriction map $v\colon \mathrm{Pic}(X)\to \mathrm{Pic}^0(T)=T^*$ is surjective and its kernel is isomorphic to $\ZZ$. Fix $x\in \PP^1$ and take a $R\in T^*$ such that $R^{\otimes k}\cong \Oo_T$ for some non-zero integer $k$. Take $L\in \mathrm{Pic}(X)$ such that $v(L) =R$. Hence $v(L^{\otimes k})=\Oo_T$. 
\end{remark}

\begin{proposition}\label{h6}
Let $E$ be a rank $r>1$ torsion free sheaf $E$ on $X$ such that $c_2(E)=0$. Then $E$ is locally free and there is a filtration $F_1\subset \cdots \subset F_r=E$ of $E$ by subbundles such that $F_{i+1}/F_i\in \mathrm{Pic}(X)$ for all $i=0,\dots ,r-1$ with the convention $F_0=0$. Moreover, $E$ is not stable for any Gauduchon metric $g$ on $X$ and it is semistable if and only if the $r$ line bundles $F_{i+1}/F_i$ have the same degree $\deg_g$.
\end{proposition}

\begin{proof}
The sheaf $E$ is locally free (Remark\thinspace \ref{h5}). Recall that $h^1(\Oo_X)= h^0(\Oo_X)=1$ and $h^2(\Oo_X)=0$.
Thus $\chi(E\otimes L) =-c_2(E\otimes L) =0$ for all line bundles $L$. Since the case $r=2$ is known, we assume $r>1$ and use induction on the rank.
\begin{itemize}

\item[(a)] In this step we assume that $E$ has a non-zero subsheaf of rank $<r$. Call $G\subset F$ a saturated subsheaf of $E$ with minimal rank $m<r$.
 We get an exact sequence
\begin{equation}\label{eqh1}
0 \to G\to  E\stackrel{v}{\to} E/G\to 0
\end{equation}
with $E/G$ torsion free of rank $r-m$. Since $E$ is locally free and $G$ is saturated in $E$, $G$ is locally free. Since $H^2(X,\ZZ)=0$, the cup product on $\mathrm{Pic}(X)$ is zero. Thus \eqref{eqh1} give $0=c_2(E) =c_2(G)+c_2(E/G)$. We get that $E/G$ is locally free and that $c_2(G)=c_2(E/G)$. The inductive
assumption of $r$ and the minimality of $m$ gives $m=1$ and that $E/G$ has a filtration $A_1\subset \cdots \subset A_{r-1} =E/G$ with $r-1$ subquotient. Taking $F_1=G$ and $F_{i+1} =v^{-1}(A_i)$ we get the filtration needed to prove the proposition.

\item[(b)] By step (a) applied to all $E\otimes L$, $L\in \mathrm{Pic}(X)$, we get that $E$ has no proper subsheaf (and hence it is stable with respect to all metrics) and in particular $h^0(E\otimes L)=0$ for all $L\in \mathrm{Pic}(X)$.  

\item[(b1)] Assume for the moment that $u_{\ast}(E\otimes L)\ne 0$ for some $L\in \mathrm{Pic}(X)$.
Since $\PP^1$ is a smooth curve and $u_{\ast}(E\otimes L)\ne 0$, there is a large positive integer $t$ such that $u_\ast(E\times L)(t)) \ne 0$
and hence $h^0(E\otimes L\otimes u^*(\Oo_{\PP^1}(t)) >0$, contradicting step (a).

\item[(b2)] By step (b1) we have $u_\ast(E\otimes L)=0$. In particular $E_{|u^{-1}(x)}$ is semistable for a general $x\in \PP^1$.
Fix a general $x\in \PP^1$ and set $J:= u^{-1}(x)$  and $Q:= E_{|J}$. We have $h^0(J,Q) =h^0(u^{-1}(x),Q)$. There is $M\in  \mathrm{Pic}^0(J)$ such
that $h^0(J,Q\otimes M)>0$; hence $h^1(J,Q\otimes M)>0$. Take $L\in \mathrm{Pic}(X)$ such that $L_{|J} =M$ (Rem.\thinspace\ref{h5.2}). Since the fibers of $u$ have dimension $1$, for any coherent sheaf $\Ff$ on $X$ we have $R^2u_{\ast}(F) =0$ and $h^2(J,F_{|J})=0$. Since $h^2(J,F_{|J})=0$, the natural map $R^2u_{\ast}(\Ff)\to H^2(J,\Ff_{|J})$ is surjective. Since $R^2u_{\ast}(\Ff)$ is locally free (it is the zero-sheaf, it has rank $0$), we get the surjectivity of the natural map $R^1u_{\ast}(\Ff)\to H^1(J,\Ff_{|J})$ (\cite{bs},  in the algebraic category it is \cite[III.12.11]{h}).
 We get $R^1u_{\ast}(E\otimes L)\ne 0$. Since $u_\ast(E\otimes L)=0$ and $R^1u_\ast(E\otimes L)$ 
 have the same rank, $R^1u_{\ast}(E\otimes L)$ is a non-zero skyscraper sheaf. 
 The spectral sequence of \eqref{eqh1} gives $\chi (E\otimes L) \ne 0$, contradicting Riemann--Roch.
\end{itemize}
\end{proof}

\begin{remark}\label{h6.0}
It is well known the for a Gauduchon metric polystability does not behave well. For a classical $X$ the degree 
$\deg_g\colon \CC^*\to \RR$ is a surjection of Lie groups and its fibers are diffeomorphic to a circle $S^1$ (so real dimension 1). Take a rank $r>1$ polystable vector bundle $E$ on $X$ such that $c_2(E)=0$.
We have $E\cong \oplus_{i=1}^{r} L_i$ with $L_i$ line bundles and $\deg_g(L_i) =\deg_g(L_1)$ for all $i$. Hence they are parametrized finite-to-one  and generically one-to-one by $(S^1)^r,$ which is an $r$-dimensional real manifold. For odd $r>1$ it has odd real dimension and hence no complex structure.
Let $\{R_\alpha\}$ be a flat family of line bundles on  $X$ with $\deg_g(R_\alpha) \ne 0$ and with $\Oo_X$ as a flat limit.
Taking $E_1\otimes R_\alpha\oplus E_1\otimes R_\alpha^\vee$ instead of the factor $E_1\oplus E_2$ we see that $E$ is a limit of unstable bundles.
\end{remark}

\begin{remark}\label{h8}
Let $E$ be a rank $r$ vector bundle on $X$. Since $H^2(X,\ZZ)=0$, the intersection form $H^2(X,\ZZ)\times H^2(X,\ZZ)\to H^4(X,\ZZ)$ is zero. Hence $c_2(E\otimes L) =c_2(E)$ for all $L\in \mathrm{Pic}(X)$. Recall that $\chi(\Oo_X)=0$. 
Hence Riemann--Roch gives $\chi(E\otimes L) =-c_2(E)$. Thus $h^1(E\otimes L)>0$ for all $L$ if $c_2(E)>0$. Note that $E$ is stable if and only if $E\otimes L$ is stable. Since $\mathrm{Pic}(X)\cong \CC^*$ as an abelian group, for any $\delta \in \mathrm{Pic}(X)$ there is $L\in \mathrm{Pic}(X)$ such that $\delta  \cong L^{\otimes r}$. Hence the map $E\mapsto E\otimes L$ induces an isomorphism between $\Mm_{r,c}$ and the moduli space of all rank $r$ stable vector bundles $F$ on $X$ such that $\det(F)\cong \delta$ and $c_2(F)=c_2$.
\end{remark}

\begin{remark}\label{h9.1}
Let $F$ be a rank $r-1$ torsion free sheaf such that $c_2(F)=1$ and $F$ is not locally free. Hence $c_2(F^{\vee \vee})=0$ (Rem.\thinspace \ref{h3}) and $F$ is obtained from $F^{\vee\vee}$ making an elementary transformation. 
Proposition\thinspace \ref{h6} gives that $F^{\vee\vee}$ is  filtrable. Hence $F$ is filtrable.
\end{remark}

Recall that a bundle $E$ is called simple if $h^0(\Hom(E,E) ) = 1$.
For the next existence result, we will use the following lemma.

\begin{lemma}\label{le1}
Let $Y$ be a connected compact complex space. Let $A$, $B$ and $C$ be vector bundles on $Y$ fitting in an exact sequence
\begin{equation}\label{eql1}
0\to A\stackrel{u}{\to} B\stackrel{v}{\to} C\to 0
\end{equation}
Assume that $A$ and $C$ are simple, $h^0(\Hom (A,C)) =h^0(\Hom(C,A)) =0$
and that \eqref{eql1} is not the zero-extension. Then $B$ is simple.

\end{lemma}
\begin{proof}
Take $f\in H^0(\End(B))$. To prove the lemma it is sufficient to prove that $f$ is induced by the multiplication by a constant.

 First assume $f(u(A)) =0$. Since \eqref{eql1} is exact, $v\circ f$ induces a morphism $h\colon C\to B$. 
 Since $v\circ h$ is an endomorphism of $C$ and $C$ is simple, there is a constant $c$ such that $v\circ h$ is the multiplication by $c$. Set $f_1\ce f-c\mathrm{Id}_B$. To prove that $f$ is the multiplication by $c$ it is sufficient to prove that $f_1=0$. Since $v\circ u =0$, $v\circ f_1$ is induced by the multiplication by $0$, i.e. $v\circ f_1=0$. Hence $f_1(B) \subseteq u(A)$.
 Thus $f_1\circ u$ induces a morphism $g: B\to A$. Since $A$ is simple, $g_{|A}\colon A\to A$ is induced by the multiplication by a constant, $z$. If $z\ne 0$, $g$ gives a splitting of \eqref{eql1}.
  Hence $z=0$. Since $v\circ f_1=0$, we get $f_1=0$.
 
 Since $f_1(u(A)) =0$, $f_1$ induces $h_1\colon C\to B$ such that 
 $v\circ h_1$ is induced by the multiplication by $0$, i.e. $v\circ h_1=0$. Hence $h_1$ induces a morphism $w: C\to A$. By assumption $w=0$. Since $v\circ h_1=0$, we get $f_1\ne 0$. 
 
 Now assume $f(u(A)) \ne 0$. First assume $v(f(A)) =0$. We get $f(A)\subseteq A$. Since $A$ is simple, $f_{|A}$ is induced by the multiplication by a constant, $d$. Set $f_2:= f-d\mathrm{Id}_B$. To conclude the proof of the lemma, it is sufficient to prove that $f_2=0$. Since $f_2(u(A)) =0$, we proved it in step (a). 
\end{proof}

\begin{theorem}\label{h10}
For each integer $r\ge 2$ there exists an irreducible rank $r$ vector bundle on $X$ with trivial determinant, $c_2=1$ and whose restriction to each fiber is semistable, i.e. with no jumps.
\end{theorem}

\begin{proof}
Since the case $r=2$ is true by \cite{BH}, we may assume $r>2$ and use induction on $r$. Let $F$ be a rank $r-1$ irreducible vector bundle such that
$\det(F)\cong \Oo_X$. Since $F$ is irreducible, it is stable with respect to any metric $\deg_g$. Take a line bundle $R$ such that $\deg_g(R) >0$. Since
$F\otimes R$ in irreducible, $h^0(F\otimes R)=0$. Since $F^*$ is irreducible, $h^0(F^*\otimes R^*\otimes \omega_X) =0$. Thus duality gives $h^2(F\otimes R)=0$.
Since Riemann--Roch gives $\chi (F\otimes R^{\otimes r}) =-c_2=-1$, there is an exact sequence
\begin{equation}\label{eqh2}
0 \to R^{1-r} \to G\to F\otimes R\to 0
\end{equation}
coming from a non-zero extension. We have $\det(G)\cong \Oo_X$. Since $F\otimes R$ is stable, it is simple. Hence Lemma \ref{le1} gives that $G$ is simple. Hence $\Ext^2(G,G)) =0$ (Remark \ref{h2}) and
$h^0(\Hom(G,G)) =1$. Riemann--Roch gives that the local deformation space of $G$ is smooth of dimension $2rc_2=2r$. By construction $G$ is filtrable.
\begin{itemize}
\item[(a)] In this step we prove that $G$ is stable. Assume that $G$ is not stable and take a destabilizing subsheaf $Q$ of $G$ with minimal rank. We may take $Q$ saturated in $G$, i.e. with $G/Q$ with no torsion. Since $X$ is a smooth surface, $Q$ is locally free and. The minimality of the rank of $Q$ gives that $Q$ is stable. Since $\deg_g(R^{1-r})<0$ and $\deg_g(G)=0$, $Q$ is nor a subsheaf of $R^{1-r}$. Hence $Q$ induces a non-zero map $h: Q\to F$. Since $\mathrm{rank}(Q)<r$
and $F$ is irreducible, we get that $h$ is injective and that $Q$ has rank $r-1$. Since $Q$ is a subsheaf of $G$ and \eqref{eqh2} is not the zero-extension, $h$ is not an isomorphism. Since $h$ is injective map between bundles with the same rank and the intersection form $H^2(X,\ZZ)\times H^2(X,\ZZ)\to H^4(X,\ZZ)$ is zero, $c_2(Q)<c_2(G)=1$, contradicting Proposition \ref{h6} and Remark \ref{le2}.

\item[(b)] In this step we prove that there is no non-zero map $G\to L$ with $L$ a line bundle. Assume the existence of $j: G\to L$ with $j\ne 0$. Thus $\mathrm{Im}(j)$ is a rank $1$ torsion free sheaf. First assume that the composition of $j$ with the injective map in \eqref{eqh2} is zero. Thus $j$ induces a surjection $j_1\colon F\to \mathrm{Im}(j)$ contradicting the unfiltrability of $F$ and the assumption $r-1>1$. Thus $\mathrm{Im}(j)$ contains a subsheaf isomorphic to $R^{1-r}$. Since $\dim X =2$ and $\mathrm{Im}(j)$ is a rank $1$ torsion free sheaf, $R^{1-r}\cong \mathrm{Im}(j)$ and $j$ induces a spitting of \eqref{eqh2}, a contradiction.

\item[(c)] As in step (b) we prove that the surjection in \eqref{eqh2} is the only non-zero map from $G$ onto a torsion free sheaf of rank $<r$. 

\item[(d)] Take a general bundle $E$ near $G$ in the deformation space of $G$ and assume that it is filtrable. Since $G$ is stable, $E$ is stable. Assume that
$E$ is not irreducible.  Let $U$ be a saturated subsheaf of $E$ with minimal rank. Since $X$ is a smooth surface, $E$ is locally free and $U$ is saturated in $E$, $U$ is locally free. The minimality of the rank of $U$ gives that $U$ is irreducible, although it may have rank $1$. Let $s$ be the rank of $U$.

\item[(d1)] Assume $s=1$. By Remark \ref{h9.1} the ones coming with $E/U$ locally free. 

\item[(d2)] Assume $2\le s\le r-2$. The limit of a family of $U$ for $E$ going to $G$ show that $F\otimes L$ is filtrable, a contradiction.

\item[(d3)] Assume $s=r-1$. In the limit with $E$ going to $G$ we get a non-zero map from a rank $r-1$ torsion free sheaf to $F$. Since $F$ is irreducible, it must be injective. Call $F_1\subseteq F$ it image. Since \eqref{eqh2} does not split, $F_1\ne F$. Hence $c_2(F_1)>1$. Thus $c_2(G)>1$, a contradiction.
\end{itemize}\end{proof}

\section{Boundary points of the moduli spaces}\label{Sloc}

Let $X$ be a smooth compact complex surface, $p\in X$, and $E$ a rank $r$ vector bundle 
 which is stable for a fixed Gauduchon metric on $X$.
For  $n>0$ let $\mathcal M_{r,n}$ denote the moduli scheme or (in the non-algebraic case) 
moduli space of rank $r$ stable vector bundles $E$ with $\det(E)\cong \mathcal O_X$ and $c_2(E) =n$. 
By Theorem \ref{f11} we have that $\mathcal M_{r,n}\ne \emptyset$.

There exists a quasi-finite surjective morphism $u\colon \mathcal V_{r,n}\to \mathcal M_{r,n}$ 
together with a {\it universal bundle} over $\mathcal V_{r,n}$, that is, a
rank $r$ vector bundle $$\mathcal E \to X\times \mathcal V_{r,n}$$
 such that for any point $[e]\in \mathcal V_{r,n}$ the corresponding bundle is isomorphic to $u(e)$. 
 Take $(E,p)\in \mathcal V_{r,n}\times X$. 
 
Note that $E\vert_{p}$ is an $r$-dimensional vector space. Let $\mathbb C_p$ 
be the skyscraper sheaf of $X$ which is zero outside $p$ and $h^0(X,\mathbb C_p)=1$.
There is a natural $\mathbb C$-linear bijection
 $$\{\mathrm{maps\, of \,}\mathcal O_X\mathrm{-sheaves\,} E\to \mathbb C_p\} \Leftrightarrow 
 \{\mathrm{linear\, maps\,} E\vert_{p} \to \mathbb C\}.$$ 
 This isomorphism preserves surjectivity.
 Hence, the set $\mathcal A$ of all surjections $E\to \mathbb C_p$ is isomorphic to $\mathbb C^r\setminus \{0\}$.
Fix $v\in \mathcal A$ and set $E':= \ker(v)$. 

\begin{definition}\cite{Ma1,Ma2} \label{elmpunto} The kernel   $E'$ is a torsion free sheaf fitting into the short exact sequence 
$$0 \longrightarrow E'\longrightarrow E\stackrel{v}{\longrightarrow}\mathbb C_p \longrightarrow 0.$$
We say that $E'$  is obtained from $E$ by an {\it elementary operation} at $p$.
\end{definition}

\begin{remark} Note that  the torsion free sheaf $E'$ satisfies $\det(E')\cong \det(E)$ 
and $$c_2(E') =c_2(E)+1.$$  We have  $E'^{\vee \vee}\cong E$ and  this
 isomorphism implies that $E$ is uniquely determined by $E'$. 
\end{remark}

Take any constant $\lambda\ne 0$. We have $\ker(\lambda v)=E'$ 
(equality as subsheaves, not just isomorphism as abstract sheaves). 
Hence, for fixed $E$ and $p$ the set of all subsheaves $E'$ is $(r-1)$-dimensional,
 parametrized by $\mathbb P^{r-1}$. 

\begin{lemma}\label{samepoint}
Let $E$ be a simple bundle on $X$ (e.g. $E$ stable) and 
let 
$$\rho_1\colon E\to \mathbb C_p\quad \text{and} \quad \rho_2\colon E\to \mathbb C_p$$ 
be non-proportional surjections. 
Then their kernels $E'$ and $E_2$ are not isomorphic as 
abstract sheaves on $X$.
\end{lemma}
 
 \begin{proof}
Take kernel sheaves $E'$ and $E_2$ associated to two non-proportional surjections 
$E\to \mathbb C_p$. The sheaves $E'$ and $E_2$ 
are certainly distinct subsheaves of $E$. We claim that $E'$ and $E_2$ are not isomorphic as 
abstract sheaves on $X$. Indeed, assume the existence of an isomorphism $f\colon E'\to E_2$. 
Since $E$ is the reflexive hull of both  $E'$ and $E_2$, $f$ induces an isomorphism 
$\hat{f}\colon  E\to E$. By assumption $\hat{f}$ is induced by the multiplication by a non-zero scalar, $c$. 
Hence $f\vert_{X\setminus \{p\}}: E_{1}\vert_{X\setminus \{p\}} \to E_{2}\vert_{X\setminus \{p\}}$ is
 induced by the multiplication by $c$. Thus $E'=E_2$ and $f$ is induced by the multiplication by $c$, a contradiction.
 \end{proof}

\begin{lemma}\label{twopoints}
 Fix a bundle $E$ over $X$.  For $p \neq q$, the kernel of  surjections  $E\to \mathbb C_p$ and $E\to \mathbb C_q$
 are non isomorphic as sheaves. 
 \end{lemma}

\begin{proof}
Take $p, q\in X$ with $p\ne q$,  and surjections  $E\to \mathbb C_p$ and $E\to \mathbb C_q$.
Let $F$ and $G$ be the kernel sheaves of $E$ associated to the surjections $E\to \mathbb C_p$ and $E\to \mathbb C_q$, 
respectively. The sheaves $F$ and $G$ are not isomorphic, because their locally free loci, 
$X\setminus \{p\}$ and $X\setminus \{q\}$, are distinct. 
\end{proof}

\begin{proposition}\label{cod2} Assume that  $E$ is a simple (e.g. stable) vector bundle over $X$ of rank $r$. 
Then the kernels of the surjections $E\rightarrow \mathbb C_p$ with $p\in X$ 
form  an $(r+1)$-dimensional family of pairwise non-isomorphic rank $r$ torsion free sheaves on $X$.
\end{proposition}

\begin{proof} Lemmas \ref{samepoint} and \ref{twopoints}.
\end{proof}

Let $X$ be a classical Hopf and $p\in X$. We  fix a Gauduchon metric on $X$.
  Let $E$ be a rank $r$ vector bundle on $X$ and $E'$ any rank $r$ torsion free sheaf fitting in an exact sequence 
\begin{equation}\label{eqa1}
0\longrightarrow E'\longrightarrow E\longrightarrow \mathbb C_p \longrightarrow 0
\end{equation}
as in definition \ref{elmpunto}.

\begin{theorem}\label{boundary}
Let $r\geq 2$. If $E$ is stable, then $E'$ is in the closure of $\mathcal M_{r,n+1}$.
\end{theorem}

\begin{proof} We assume that $E$ is stable, therefore, by Remark \ref{h3} $E'$ is stable. 
We have $\det(E') =\det(E)$ and $c_2(E')=c_2(E)+1$. Set $n:= c_2(E)$. We have that
$\Ext^2(E,E)=0$ by the stability of $E$. If $E'$ is a flat limit of locally free sheaves, 
then nearby sheaves are stable.
Hence, as soon as we prove that $E'$ is a limit of a sequence $F_\alpha$ 
of vector bundles (with rank $r$ and $c_2=n+1$) then,
 by the openness of stability, it follows that $E'$ is in the closure of $\mathcal M_{r,n+1}$.

Since stability of torsion free sheaves is an open condition  for any Gauduchon metric, to prove that
$E'$ is in the closure of $\mathcal M_{n+1}$ (a codimension $1$ boundary), 
it is sufficient to prove that, for each $F_\alpha$, the local deformation space of $F_\alpha$
 is smooth of dimension $\dim \mathcal M_{n+1}$. To prove this,  
 it is sufficient to show that $\Ext^2(F_\alpha,F_\alpha) =0$ (Global Ext). 
 Note that $F_\alpha$ locally deforms either to another $F_{\alpha'}$ or else to a locally free sheaf. 
 
 We use the local-to-global spectral sequence of $\Ext$ (\cite[Th.\thinspace 7.3.3]{god},\cite[Th.\thinspace 12.13]{McC}).
Write $F_\alpha =E'$ with $E'$ as in \eqref{eqa1}. Since $E'$ is stable, it is simple, i.e. $\dim \Ext^0(E',E')=1$, we have
 $$\dim \Ext^0(E',E')-\dim \Ext^1(E',E') +\dim \Ext^2(E',E')=1-(2rn+2r).$$ 
 Hence, knowing that  
$\Ext^2(E',E')=0$ would conclude the proof. 

We write $\lext_{\mathcal O_X}$ for the local $\Ext$ sheaf. 
Each sheaf $\lext^i_{\mathcal O_X}$ is coherent. Since $E'$ is locally free outside $p$, for all $i\ge 1$ the coherent sheaves $\lext^i_{\mathcal O_X}(E',E')$ have finite support and hence 
$h^j(\lext^i_{\mathcal O_X}(E',E'))=0$ for all $j>0$ and all $i>0$. 
Certainly, for any coherent sheaf $\mathcal F$ and any $i\in \mathbb N$ all 
$\lext^i_{\mathcal O_X}(\mathbb C_p,\mathcal F)$ and
all  $\lext^i_{\mathcal O_X}(\mathcal F,\mathbb C_p)$ 
are sheaves which, if non-zero, have p as their support, i.e. their restriction to $X\setminus \{p\}$ is zero.
We have $\lhom_{\mathcal O_X}(\mathbb C_p,E') = \lhom_{\mathcal O_X}(\mathbb C_p,E) = 0$, 
because $E'$ and $E$ have no torsion. If $\mathcal F$ is locally free, then 
$\lext^i_{\mathcal O_X}(\mathcal F,\mathcal G)=0$ for all $i>0$ and all coherent sheaves $\mathcal G$, while $\lext^0_{\mathcal O_X}(\mathcal F,\mathcal G)=\mathcal G^r$
if $\mathcal F$ has rank $r$ and $\mathcal G$ has finite support.

By \cite{Schu}, every coherent sheaf on $X$ has a finite resolution given by holomorphic vector bundles.
Accordingly, we get that
$\Ext^i(\mathcal F,\mathcal G)=0$ for all $i\ge 3$ and all coherent sheaves $\mathcal F$ and $\mathcal G$.

Since each local ring $\mathcal O_{X,p}$ of $X$ is a $2$-dimensional regular local ring, 
$\Ext^i_{\mathcal O_X}(\mathcal F,\mathcal G)=0$ for all $i\ge 3$ and 
all coherent sheaves $\mathcal F$ and $\mathcal G$.

Thus, applying the functor $\Hom(-,E')$ we get the exact sequence
\begin{align}\label{eqa3}
&0\to \Hom(\mathbb C_p,E')\to \Hom(E,E')\to \Hom(E',E')\notag\\
&\to \Ext^1(\mathbb C_p,E')\to \Ext^1(E,E')\to \Ext^1(E',E')\\
&\to \Ext^2(\mathbb C_p,E')\to \Ext^2(E,E')\to \Ext^2(E',E')\to 0\notag \text{.}
\end{align}
Recall that to conclude the proof  it is sufficient to show that 
$\Ext^2(E',E')=0$.
 By \eqref{eqa3} it would be sufficient to prove that $\Ext^2_{\mathcal O_X}(E,E')=0$.

 Applying the functor $\Hom(E,-)$ to \eqref{eqa1} 
we get the long exact sequence
\begin{align}\label{eqa5}
&0\to  \Hom(E,E')\to  \Hom(E,E)\to  \Hom(E,\mathbb C_p)\notag\\
&\to \Ext^1(E,E')\to \Ext^1(E,E)\to \Ext^1(E,\mathbb C_p)\\
&\to \Ext^2(E,E')\to \Ext^2(E,E)\to \Ext^2(E,\mathbb C_p)\to 0\notag
\end{align}
Since $\Ext^2(E,E)=0$, \eqref{eqa5} shows that to conclude the proof it would be sufficient to prove that 
$\Ext^1(E,\mathbb C_p)=0$.

Since $E$ is locally free, 
we have $\lext^i_{\mathcal O_X}(E,\mathbb C_p)=0$ for all $i\ge 1$.
Since $\mathbb C_p$ is supported on $p$, $\lhom_{\mathcal O_X}(E,\mathbb C_p)$ is supported on $p$. 
Hence $h^i(\lhom_{\mathcal O_X}(E,\mathbb C_p))=0$ for all $i\ge 1$.
Since $h^1(\lhom_{\mathcal O_X}(E,\mathbb C_p))=0$, $h^0(\lext^1_{\mathcal O_X}(E,\mathbb C_p))=0$ and $h^0(\lext^2_{\mathcal O_X}(E,\mathbb C_p))=0$, the local-to-global spectral sequence \cite[p.\thinspace 5]{McC}
gives $\Ext^1(E,\mathbb C_p)=0$,  therefore it follows that $\Ext^2(E',E')=0$.
\end{proof}

 Let $\Gamma\subset \Nn_{r,n}\times X\times \PP(\Nn_{r,n})$ denote the set of all triples
$(F,p,v)$, where $F\in \Nn_{r,n}$, $p\in X$ and $v\in \PP(F^\ast_{|p})$. Let $v_1$ be an element of the $r$-dimensional vector space $F^\ast_{|p}$ inducing $v$. We have $v_1\ne 0$ and any $2$ elements of $F^\ast_{|p}$ inducing $v$ are proportional. Since $v_1\ne 0$, $v_1\colon F_{|p}\to \CC$ is a surjection. Hence $v$
uniquely determines the rank $r$ torsion free sheaf $F'\subset F$. We proved that $\Gamma$ is equipped with a family of rank $r$ torsion free sheaves
with $c_2=n+1$, a unique non-locally free point, smooth deformation spaces 
(more precisely, $\Ext^2(X;F',F')=0$) and that each $F'\in \Gamma$ is a flat limit
of elements of $\Mm_{r,n+1}$. We have $\dim \Gamma =\dim \Nn_{r,n}+r+1$. 
We also proved that two isomorphic sheaves $F',F''\in \Gamma$ come from the same $F\in \Mm_{r,n}$, 
the same $p\in X$ and the same $v\in \PP(F^\ast_{|p})$.
 Hence, if we take  $\Mm_{r,n}$ as a neighborhood of $E$, then $\dim \Gamma =2rn+r+1$. 
 Since we started with $E\in \Mm_{r,n}$, $n>0$, all of these torsion free but not locally free sheaves
 have $c_2\ge 2$. This is sharp for the following reason.
 
\begin{proposition}\label{gat1}
Let $E$ be a  stable torsion-free sheaf of rank $r \geq 2$  on $X$ 
 and such that $c_2(E)=1$. Then $E$ is locally free.
\end{proposition}

\begin{proof}
Assume that $E$ is not locally free, i.e. assume that $A\ce E^{\vee\vee}/E$ is a non-zero sheaf. Since $X$ is a smooth surface, $A$ is supported by finitely many points. Since $A\ne 0$, $z:= h^0(X,A)>0$. We have $c_2(E^{\vee\vee})=1-z\le 0$. Hence $E^{\vee\vee}$ is not stable, contradicting \cite[Prop.\thinspace 5.27]{brin} applied twice (first it gives that $E^\vee$ is stable, because $E$ is stable, then it gives that $E^{\vee\vee}$ is stable, see Rem.\thinspace\ref{h3}). \end{proof}

\paragraph{\bf Acknowledgements}  
We are grateful to Akiyoshi Sannai of Kyoto University for having pointed out 
that corrections to the earlier version of this text were needed, as well as for  
suggesting  improvements.
 E. Ballico is a member of  GNSAGA of INdAM (Italy).

\end{document}